\documentclass[12pt]{amsart}

\usepackage{amsfonts,amssymb,stmaryrd,amscd,amsmath,latexsym,amsbsy}

\newtheorem{theorem}{Theorem}[section]
\newtheorem{lemma}[theorem]{Lemma}

\theoremstyle{definition}

\newtheorem{remark}[theorem]{Remark}

\newcommand{\End}{\text{End}}
\newcommand{\Hom}{\text{Hom}}

\newcommand{\ben}{\begin{enumerate}}
\newcommand{\een}{\end{enumerate}}

\theoremstyle{plain}

\newtheorem*{sol}{Solution}

\theoremstyle{definition}

\theoremstyle{remark}

\newcommand{\solu}[1]{\begin{sol}{\bf (\ref{#1})}}

\def\End{\mathrm{End}}
\def\Hom{\mathrm{Hom}}

\begin{document}

\title{Semisimple and $G$-equivariant simple algebras over operads}

\author{Pavel Etingof}

\begin{abstract}
Let $G$ be a finite group. There is a standard theorem on the classification of $G$-equivariant finite dimensional simple commutative, associative, and Lie algebras
(i.e., simple algebras of these types in the category of representations of $G$). Namely, such an algebra is of the form $A={\rm Fun}_H(G,B)$, where $H$ is a subgroup of $G$, and $B$ is a 
simple algebra of the corresponding type with an $H$-action. We explain that such a result holds in the generality of algebras over a linear operad. 
This allows one to extend Theorem 5.5 of \cite{S} on the classification of simple commutative algebras in the Deligne category 
${\rm Rep}(S_t)$ to algebras over any finitely generated linear operad. 
\end{abstract}

\maketitle 

\section{Semisimple algebras over operads}

\subsection{Algebras} 
Let $C$ be a linear operad over a field $F$ (\cite{LV}). 
E.g., $C$ can be the operad of commutative associative unital algebras, associative unital algebras, or Lie algebras (the latter if $1/2\in F$). 

Recall (\cite{LV}) that a $C$-{\it algebra} is a vector space $A$ over $F$ with a collection of linear maps $\alpha_n: C(n)\to \Hom_F(A^{\otimes n},A)$ compatible with the operadic structure.  
Clearly, a direct sum of finitely many $C$-algebras is a $C$-algebra. 

Given a $C$-algebra $A$, we can define the space $E_A\subset {\rm End}_F(A)$ spanned over $F$ 
by operators of the form $\alpha_n(c)(a_1,...,a_{j-1},?,a_j,...,a_{n-1})$ for various $n\ge 2$, $c\in C(n)$, 
and $a_i\in A$. By the definition of an operad, $E_A$ is a (possibly non-unital) subalgebra of 
${\rm End}_F(A)$. We also denote by $L_A$ the image of $C(1)$ in $\End_F(A)$. Clearly, 
$L_A$ is a unital subalgebra and $L_AE_A=E_AL_A=E_A$. Thus $R_A:=L_A+E_A$ is a unital subalgebra of ${\rm End}_F(A)$, and $E_A$ is an ideal in $R_A$. 

\begin{lemma}\label{dirsum} One has $E_{A_1\oplus...\oplus A_m}=E_{A_1}\oplus...\oplus E_{A_m}$.  
\end{lemma} 

\begin{proof} It is clear that $E_{A_1\oplus...\oplus A_m}\subset E_{A_1}\oplus...\oplus E_{A_m}$.
Let $a_i\in A_r$, $c\in C(n)$, and $b=\alpha_n(c)(a_1,...,a_{j-1},?,a_j,...,a_{n-1})\in E_{A_r}$. Let 
$b':=(0,...,b,...,0)$ (where $b$ is at the $r$-th place). Then we have \linebreak $b'=\alpha_n(c)(a_1',...,a_{j-1}',?,a_j',...,a_{n-1}')$, where $a_i'=(0,...,a_i,...,0)$. 
Hence $b'\in E_{A_1\oplus...\oplus A_m}$. 
Thus $E_{A_1\oplus...\oplus A_m}\supset E_{A_1}\oplus...\oplus E_{A_m}$.
\end{proof} 

\subsection{Ideals} 
By an {\it ideal} in a $C$-algebra $A$ we mean a subspace $I\subset A$ such that
for any $n\ge 1$, $c\in C(n)$, $j\in [1,n]$, and $T\in A^{\otimes j-1}\otimes I\otimes A^{\otimes n-j}$ 
one has $\alpha_n(c)T\in I$. 

\begin{lemma}\label{submo} (i) $I\subset A$ is an ideal if and only if it is an $R_A$-submodule
of $A$.

(ii) $A=A_1\oplus...\oplus A_m$ as an $R_A$-module if and only if it is so as a $C$-algebra. 
\end{lemma} 

\begin{proof} (i) This follows directly from the definition. 

(ii) The ``if'' direction is clear. To prove the ``only if'' direction, 
note that by (i) $A_i$ are ideals in $A$, hence $\alpha_n(...,x,...,y,...)=0$ 
once $x\in A_i$ and $y\in A_j$ with $j\ne i$, which implies the statement. 
\end{proof}  

It is clear that if $I\subset A$ is an ideal then $A/I$ is a $C$-algebra, and 
$E_{A/I}, L_{A/I}, R_{A/I}$ are homomorphic images of $E_A, L_A,R_A$ in ${\rm End}_F(A/I)$. 
 
\subsection{Simple and semisimple algebras} 
 From now on we assume that $A$ is a finite dimensional $C$-algebra. We say that $A$ is {\it simple} if any ideal in $A$ is either $0$ or $A$ (i.e., $A$ is a simple $R_A$-module), and 
$E_A\ne 0$.\footnote{Note that this recovers the standard definition for commutative, associative, and Lie algebras. Moreover, while
in the commutative and associative case, the condition $E_A\ne 0$ is automatic for $A\ne 0$ because of the unit axiom, in the Lie case it 
is needed (as an abelian Lie algebra is not simple). Note also that if $C(n)=0$ for $n\ne 1$ (i.e., when $C$ is an ordinary algebra), 
then $E_A=0$ automatically, so there are no simple $C$-algebras, even though there may exist simple $C$-modules.} 

\begin{lemma}\label{simea} If $A$ is a simple $C$-algebra then $E_A=R_A$, and it is a central simple algebra (over some finite field extension of $F$). 
\end{lemma}  

\begin{proof} Since $A$ is a faithful simple $R_A$-module, $R_A$ is central simple. 
Since $E_A\ne 0$ and $E_A$ is an ideal in $R_A$, we have $E_A=R_A$. 
\end{proof} 

We say that $A$ is {\it semisimple} if $A$ is a direct sum of a finite (possibly empty) collection 
of simple $C$-algebras: $A=A_1\oplus...\oplus A_m$.  

\begin{lemma} \label{ideals} Let $A=A_1\oplus...\oplus A_m$ be a semisimple $C$-algebra with simple constituents $A_i$. 
Then the only ideals in $A$ are $\oplus_{i\in S}A_i\subset A$, where $S\subset [1,m]$.  
\end{lemma} 

\begin{proof} Clearly, the subspaces in the lemma are ideals. Conversely, 
let $I\subset A$ be an ideal. Let $a=(a_1,...,a_m)\in I$. By Lemma \ref{dirsum} and Lemma \ref{simea}, 
the projection operator $P_i: A\to A$ to $A_i$ along $\oplus_{j\ne i}A_j$ is contained in $E_A$.  
Thus, $P_ia=(0,...,a_i,...,0)\in I$. This implies the statement.   
\end{proof} 

\subsection{The radical}
Let $A'$ be the maximal semisimple quotient of $A$ as an $R_A$-module (it exists by the standard theory of finite dimensional algebras). 
Let $\overline{A}$ be the quotient of $A'$ by the kernel of the action of $E_A$ (which is an $R_A$-submodule of $A$). 
Define the {\it radical} ${\rm Rad}(A)$ of $A$ to be the kernel of the projection of $A$ onto $\overline{A}$. 
So the radical of $A/{\rm Rad}(A)=\overline{A}$ is zero. In particular, if $A$ is a semisimple $C$-algebra, then ${\rm Rad}(A)=0$. 

\begin{theorem}\label{th1} (i) $\overline{A}$ is a semisimple $C$-algebra. In particular, 
${\rm Rad}(A)=0$ if and only if $A$ is semisimple. 

(ii) If $I\subset A$ is an ideal, then $A/I$ is a semisimple $C$-algebra if and only if $I$ contains ${\rm Rad}(A)$. 
\end{theorem} 

\begin{proof} (i) By the definition, $\overline{A}$ is a semisimple $R_A$-module, such that $E_A$ acts by nonzero 
on all its simple summands. Hence by Lemma \ref{submo}(ii), $\overline{A}$ is a semisimple $C$-algebra. 

(ii) The ``if'' direction holds by (i) and Lemma \ref{ideals}.
To prove the ``only if'' direction, let $I\subset A$ be an ideal such that 
$A/I$ is a semisimple $C$-algebra: $A/I=A_1\oplus...\oplus A_m$. 
Then by Lemma \ref{submo}(ii) $A/I$ is a semisimple $R_{A/I}$-module and hence $R_A$-module, 
with simple constituents $A_i$, and the action of $E_A$ on $A_i$ is nonzero. 
Thus $I \supset {\rm Rad}(A)$. 
\end{proof} 

\section{$G$-equivariant simple algebras over operads} 

Now let $G$ be a finite group, and $A$ be a $C$-algebra with an action of $G$. 
Let us say that $A$ is a {\it simple $G$-equivariant $C$-algebra} if the only $G$-invariant ideals 
in $A$ are $0$ and $A$, and $E_A\ne 0$. 

\begin{lemma}\label{semi} (i) If $B$ is a simple $C$-algebra then we have ${\rm Aut}(B^{\oplus n})=S_n\ltimes {\rm Aut}(B)^n$. 

(ii) If $A$ is a simple $G$-equivariant $C$-algebra then $A$ is semisimple 
as a usual $C$-algebra. Moreover, $G$ acts transitively on 
the simple constituents of $A$, and in particular they are 
all isomorphic.  
\end{lemma} 

\begin{proof} (i) Clearly, $S_n\ltimes {\rm Aut}(B)^n$ acts on $B^{\oplus n}$, so we need to show that any automorphism 
$g$ of $B^{\oplus n}$ belongs to this group. By Lemma \ref{ideals}, the minimal (nonzero) ideals of 
$B^{\oplus n}$ are the $n$ copies of $B$. So they must be permuted by $g$, inducing an element $s\in S_n$. 
Thus $gs^{-1}$ is an automorphism preserving all the copies of $B$. So $gs^{-1}\in {\rm Aut}(B)^n$, as desired. 

(ii) Let $I$ be kernel of the projection from $A$ to its maximal semisimple quotient $A'$ as an 
$R_A$-module. Then by Lemma \ref{submo}(i), $I$ is a $G$-invariant ideal in $A$, and $I\ne A$. 
Hence $I=0$, and $A$ is a semisimple $R_A$-module. So by Lemma \ref{submo}(ii), 
$A=A_1\oplus...\oplus A_m$ is a semisimple $C$-algebra. Thus 
by Lemma \ref{ideals}, the minimal ideals of 
$A$ are the $A_i$. So they are permuted by $G$. Moreover, the action of $G$ on these 
ideals must be transitive, as every orbit gives a nonzero $G$-invariant ideal. 
\end{proof} 

Now let $B$ be a simple $C$-algebra, $H$ a subgroup of $G$, and $\phi: H\to {\rm Aut}(B)$ a homomorphism. 
Let $A={\rm Fun}_H(G,B)$ be the space of $H$-invariant functions on $G$ with values in $B$. 
Then it is clear that $A$ has a natural structure of a simple $G$-equivariant $C$-algebra, isomorphic to $B^{\oplus |G/H|}$ as a 
usual $C$-algebra. Note that the stabilizer of any minimal ideal of $A$ is a subgroup of $G$ conjugate to $H$.  

\begin{theorem}\label{semi1} Any simple $G$-equivariant $C$-algebra $A$ is of the form $A={\rm Fun}_H(G,B)$. Moreover, the subgroup $H$ is defined by $A$ 
uniquely up to conjugation in $G$, and $\phi$ is defined uniquely up to conjugation in ${\rm Aut}(B)$.  
\end{theorem}  

\begin{proof}
By Lemma \ref{semi}(ii), $G$ acts transitively on the set of minimal  ideals in $A$, and they are 
all isomorphic to some simple $C$-algebra $B$. Thus, the result follows from Lemma \ref{semi}(i) and the standard classification of homomorphisms 
$G\to S_n\ltimes {\rm Aut}(B)^n$. Namely, let $H$ be the stabilizer of one of the copies of $B$.  
Then $H$ acts on $B$ through some homomorphism $\phi: H\to {\rm Aut}(B)$. Moreover, we have a canonical $G$-equivariant linear map 
$\psi: A\to {\rm Fun}_H(G,B)$ corresponding via Frobenius reciprocity to the $H$-stable projection $A\to B$ to the chosen copy of $B$ along the direct sum of all the other copies. It is easy to check using Lemma \ref{semi} that 
$\psi$ is an isomorphism of $G$-equivariant $C$-algebras. The rest is easy.  
\end{proof} 

\begin{remark}
1. Note that in the examples of commutative, associative, 
and Lie algebras we obtain the classical theorems 
about classification of simple $G$-equivariant algebras of these types. 

2. Lemma \ref{semi} and Theorem \ref{semi1} don't hold without the assumption $E_A\ne 0$. E.g., one may take $A$ to be any irreducible representation of $G$ 
equipped with the zero Lie bracket.    

3. The results of this section extend verbatim to the case when $G$ is any group (not necessarily finite), 
or is an affine algebraic group over $F$. Namely, as in the finite group case, the classification of simple $G$-equivariant algebras reduces to classification of homomorphisms $G\to S_n\ltimes {\rm Aut}(B)^n$, which are paramertized by finite index subgroups $H$ of $G$ and homomorphisms $\phi: H\to {\rm Aut}(B)$ up to conjugation. 
\end{remark} 

\begin{remark} While the question of classification of 
$G$-equivariant simple algebras over operads 
is natural in its own right, the motivation for writing this note was to provide a more general context for the results of \cite{S}. 
Namely, Lemma \ref{semi} and Theorem \ref{semi1} allow one to extend the main result of \cite{S} 
(Theorem 5.5 on the classification of simple commutative algebras in the Deligne category ${\rm Rep}(S_t)$) to
algebras over a finitely generated linear operad $C$ over $\Bbb C$. Informally speaking, this generalization says that for transcendental $t$ any 
such algebra is obtained by induction from ${\rm Rep}(G)\boxtimes {\rm Rep}(S_{t-k})$ of an interpolation $B$ of 
a family of $G\times S_{n-k}$-equivariant simple algebras $B_n$, defined for some strictly increasing sequence of positive integers $n$ and depending algebraically on $n$. 

This gives a classification of simple $C$-algebras in ${\rm Rep}(S_t)$ whenever a classification of ordinary simple $C$-algebras (and their automorphisms) is available.   
For instance, in the case of associative unital algebras, $B={\rm End}(V)$, where $V$ is an object of ${\rm Rep}(S_t)$, and 
in the case of Lie algebras $B=\mathfrak{sl}(V), \mathfrak{o}(V)$, or $\mathfrak{sp}(V)$, where in the second case $V$ is equipped 
with a nondegenerate symmetric form and in the third case with a nondegenerate skew-symmetric form. 

The proof of this generalization is similar to the proof of Theorem 5.5 of \cite{S}, which 
covers the case of commutative unital algebras (in which case $B=\Bbb C$), but is somewhat more complicated 
since in general ${\rm Aut}(B)\ne 1$. The finite generation assumption for $C$ 
is needed to validate the constructibility arguments of \cite{S}, Section 4. 
This will be discussed in more detail elsewhere. 
\end{remark} 

{\bf Acknowledgements.} The author is grateful to Nate Harman for useful discussions. 
 The work of the author was partially supported by the NSF grant DMS-1502244.

\end{document}